\documentclass[11pt]{article}
\usepackage{epsf}
\usepackage{epsfig}
\usepackage{enumerate}
\usepackage{url}

\usepackage{amsmath,amsfonts,amssymb,amsthm}
\usepackage{wasysym}
\usepackage[retainorgcmds]{IEEEtrantools}

\usepackage{hyperref}
\def\cM{{\cal M}}
\def\cO{{\cal O}}
\def\cP{{\cal P}}
\def\cT{{\cal T}}
\newcommand{\iref}[1]{(\ref{#1})}
\def\lsim{\hbox{\kern -.2em\raisebox{-1ex}{$~\stackrel{\textstyle<}{\sim}~$}}\kern -.2em}
\def\gsim{\hbox{\kern -.2em\raisebox{-1ex}{$~\stackrel{\textstyle>}{\sim}~$}}\kern -.2em}
\def\R{{\rm \hbox{I\kern-.2em\hbox{R}}}}
\def\Z{{\rm {{\rm Z}\kern-.28em{\rm Z}}}}
\def\sm{\setminus}
\def\<{\langle}
\def\>{\rangle}
\newcommand{\be}{\begin{equation}}
\newcommand{\ee}{\end{equation}}
\newcommand\trans{\mathrm T}
\newtheorem{theoremIntro}{Theorem}
\newtheorem{theorem}{Theorem}[section]
\newtheorem{definition}[theorem]{Definition}
\newtheorem{prop}[theorem]{Proposition}
\newtheorem{lemma}[theorem]{Lemma}
\newtheorem{corollary}[theorem]{Corollary}
\newtheorem{remark}[theorem]{Remark}

\def\ve{\varepsilon}
\DeclareMathOperator\Id{Id}

\DeclareMathOperator\length{length}

\DeclareMathOperator\dist{d}
\DeclareMathOperator\distC{D}

\begin{document}
\title{On the Accuracy of Anisotropic Fast Marching} 
\author{Jean-Marie Mirebeau
\footnote{CNRS, University Paris Dauphine, UMR 7534, Laboratory CEREMADE, Paris, France.}
\footnote{Part of this research was conducted when the author was visiting Dr Y. Babenko, at Kennesaw state university, funded by Simons Collaboration Grant $\#210363$.}
}
\maketitle
\date{}
\begin{abstract}
The fast marching algorithm, and its variants, solves numerically the generalized eikonal equation associated to an underlying riemannian metric $\cM$. 
A major challenge for these algorithms is the non-isotropy of the riemannian metric, which magnitude is characterized  by the anisotropy ratio $\kappa(\cM) \in [1,\infty]$. 
Applications of the eikonal equation to image processing \cite{BC10,JBTDPIB08} often involve large anisotropy ratios, which motivated the design of new algorithms. %

A variant of the fast marching algorithm, introduced in \cite{M12}, addresses the problem of large anisotropies using an algebraic tool named lattice basis reduction. 
The numerical complexity of this algorithm is insensitive to anisotropy, under extremely weak assumptions.
We establish in this paper, in the simplified setting of a constant riemannian metric, that the accuracy of this algorithm is also extremely robust to anisotropy : in an average sense, it does not degrade as $\kappa(\cM)$ increases. We also extend this algorithm to higher dimension.
\end{abstract}

%


\section*{Introduction}

The Generalized Eikonal Equation is a Partial Differential Equation (PDE), which characterizes the riemannian distance, associated to a riemannian metric $\cM$, between a domain's boundary and an arbitrary point of this domain. Alternatively this PDE is also the level set formulation of an elementary front propagation model \cite{S99}, where the front speed is dictated locally by the front position and orientation, a dependence encoded in the riemannian metric $\cM$, but independent of global properties of the front, or higher order properties such as its curvature. This restrictive setting allows to compute numerically the front evolution using the 
fast marching algorithm or its variants, which are derivatives of Dijkstra's shortest path algorithm, instead of more costly level set methods. 
Solutions to the Generalized Eikonal equation have numerous applications \cite{S99}, including medical image  processing \cite{BC10,JBTDPIB08} which motivates this work. 
In this context, pronounced anisotropies are not uncommon, which challenges currently available algorithms. 

The original fast marching algorithm \cite{T95} is limited to isotropic metrics. Several variants were later developed, which can handle limited \cite{JBTDPIB08} or arbitrary \cite{SV03,M12} anisotropies. Anisotropy magnitude is characterized by the anisotropy ratio $\kappa(\cM)$, see \iref{defKappa}, which reflects the local distortion of distances. The method \cite{M12} was developed for applications based on a cartesian grid, such as image processing, where robustness to large anisotropies is crucial, both in terms of computational complexity and numerical accuracy. This algorithm was originally limited to two and three dimensional domains, but we extend it to four dimensions in the last section of this paper.
The complexity $\cO(N \ln N+ N \ln \kappa(\cM))$ of this algorithm, where $N$ denotes the cardinality of the discrete computational domain, is similar to that of the original isotropic fast marching algorithm $\cO(N \ln N)$, under the weak assumption that the number of discretization points $N$ exceeds the anisotropy ratio $\kappa(\cM)$ (in typical applications $\kappa(\cM) \lsim 100$ and $N \gsim 10^4$). 

We study in this paper the accuracy of this algorithm, motivated by numerical experiments, 
which showed significant improvements over an alternative method \cite{BR06} for solving the generalized eikonal equation. More precisely, the highly anisotropic benchmark ($\kappa(\cM)=100$) presented on Figure 3 in \cite{M12} shows an error reduction by a factor 7 on a $1200\times 1200$ grid, while the computation time is reduced by a factor 200. 
The error analysis in the case of an arbitrary continuous riemannian metric $\cM$ on a general domain was not succeeded, and we therefore restrict our attention to case of a constant metric $\cM$ on the domain $\R^d\sm \{0\}$. While this simplified setting is purely academic, the author believes that it gives a valuable insight on the local behavior of the general case. Note that the fast marching algorithm, and most of its variants, are first order discretizations of the eikonal equation PDE, but that we are interested the precise dependence of the numerical error with respect to the anisotropy ratio $\kappa(\cM)$. In the worst case scenario, our error bound grows like a power of $\kappa(\cM)$. We show however in Theorems \ref{th:estim} and \ref{th:mu} that, in an average sense over all space directions, the numerical error in independent of the anisotropy ratio $\kappa(\cM)$.\\

In order to state our results, we need to introduce some notations. We consider a fixed integer $d \geq 1$, and we denote by $S_d^+$ the collection of $d\times d$ symmetric positive definite matrices. For each $M\in S_d^+$ we introduce the scalar product $\<\cdot, \cdot\>_M$ and the norm $\|\cdot\|_M$ defined for all $u,v \in \R^d$ by 
\be
\label{defScalNorm}
\<u,v\>_M := u^\trans M v, \qquad \|u\|_M := \sqrt{\<u,u\>_M}.
\ee
Consider an open domain $\Omega \subsetneq \R^d$, and a riemannian metric $\cM \in C^0( \overline \Omega , S_d^+)$. 
The generalized eikonal equation is the following PDE :
\be
\label{eikonal}
\left\{
\begin{array}{rl}
\|\nabla \distC(z)\|_{\cM(z)^{-1}} = 1 & \text{for all } z \in \Omega,\\
\distC(z) = 0 & \text{for all } z\in \partial \Omega.
\end{array}
\right.
\ee
We chose null boundary conditions for simplicity, and we refer the reader interested in more general boundary conditions to the discussion in \cite{BR06}. 
The PDE \iref{eikonal} characterizes the distance function $\distC \in C^0(\overline \Omega, \R_+)$ to the boundary of $\Omega$ : the unique viscosity solution \cite{L82} is 
\begin{eqnarray*}
\distC (z) &=& \inf \{ \length(\gamma); \, \gamma\in C^1([0,1], \overline \Omega), \, \gamma(0)=z, \, \gamma(1) \in \partial \Omega\},\\
\length(\gamma) &:=& \int_0^1 \| \gamma'(t)\|_{\cM(\gamma(t))} dt.
\end{eqnarray*}

The anisotropy ratio $\kappa(\cM)\in [1, \infty]$ is the supremum value of $\|u\|_{\cM(z)} / \|v\|_{\cM(z)}$, where $z\in \overline \Omega$ and $u,v\in \R^d$ are vectors of unit euclidean norm. In other words
\be
\label{defKappa}
\kappa(\cM) := \sup_{z\in \overline\Omega} \kappa(\cM(z)), \quad \text{ where for } M \in S_d^+, \ \kappa(M) := \sqrt{\|M\| \|M^{-1}\|}. 
\ee
The discretization of the PDE \iref{eikonal} takes the form of a fixed point problem \cite{T95,BR06,M12}
\be
\label{eikonal_disc}
\left\{
\begin{array}{ll}
\dist(z) = \Lambda(\dist, z) & \text{for all } z \in \Omega_*,\\
\dist(z) = 0 & \text{for all } z\in \partial \Omega_*,
\end{array}
\right.
\ee
where $\Omega_*$ and $\partial \Omega_*$ denote discrete sets devoted to the sampling of the continuous domain $\Omega$ and its boundary $\partial \Omega$ respectively, and $\dist : \Omega_* \cup \partial \Omega_* \to \R_+$ denotes a discrete map. The Hopf-Lax update operator $\Lambda(\dist, z)$ can take several forms, but depends only on the value of $\dist(x)$ for a finite number of points of $x\in \Omega_* \cup \partial \Omega_*$, referred to as the neighbors of $z$. Under suitable assumptions, the fixed point problem \iref{eikonal_disc} can be solved by iterative methods \cite{BR06}. If in addition the Hopf-Lax operator $\Lambda$ satisfies a \emph{causality property} (see Lemma \ref{lem:upwind}), then the fast marching algorithm \cite{T95} can be applied to decouple \iref{eikonal_disc} and solve it in a ``single pass'', using a clever ordering of the set $\Omega_*$.\\

In order to make a sharp error analysis of this numerical scheme, we restrict as announced our attention to a very specific (and academic) setting. We consider the domain $\Omega = \R^d \sm \{0\}$, thus $\partial \Omega = \{0\}$, which will be discretized on the cartesian grid : $\Omega_* := \Z^d\sm \{0\}$, $\partial \Omega_* := \{0\}$. 
We fix a matrix $M\in S_d^+$ and we consider the constant riemannian metric defined by $\cM(z) = M$ for all $z\in \R^d$. We denote by $\distC_M$ the solution of \iref{eikonal}, which has the explicit expression 
\be
\label{exact_sol}
\distC_M(z) = \|z\|_M, \quad z\in \R^d.
\ee

We use the variant introduced in \cite{M12} of the Hopf-Lax update operator $\Lambda$, which is based on a $M$-reduced mesh $\cT$ (see Definition 1.1 in \cite{M12}), fixed in the rest of this paper. In other words $\cT$ is a finite conforming mesh which satisfies the following properties

(I)\phantom{II} The union of the simplices $T \in \cT$ is a neighborhood of the origin.

(II)\phantom{I} The vertices of each simplex $T\in \cT$ lie on the lattice $\Z^d$, and $T$ has volume $1/d!$.

(III) For each $T\in \cT$, one of the vertices of $T$ is the origin $0$, and the others denoted by $v_1, \cdots, v_d$ satisfy for all $1\leq i,j \leq d$ the acuteness condition
\be
\label{acuteness}
\<v_i, v_j\>_M \geq 0.
\ee

\begin{figure}
\begin{center}
\includegraphics[width=5cm]{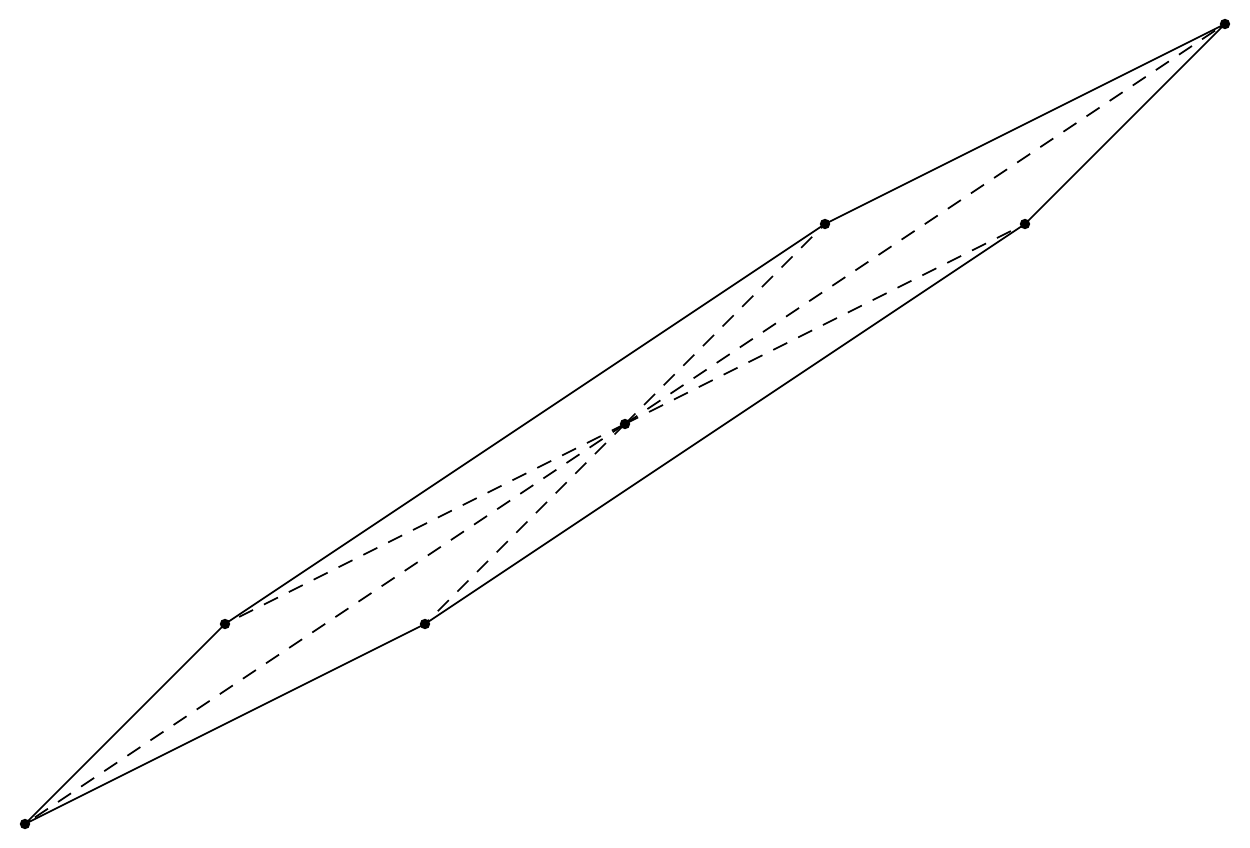}
\hspace{-0.2cm}
\includegraphics[width=5cm]{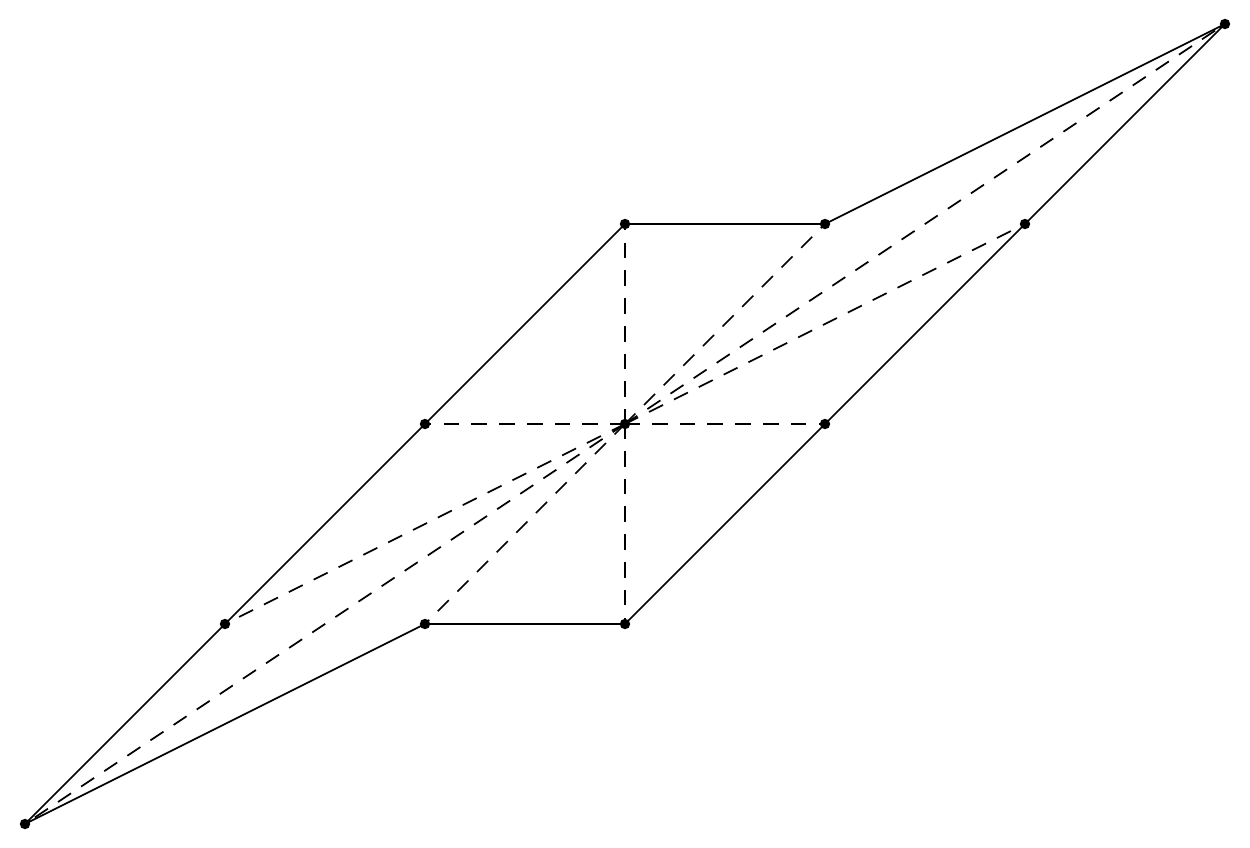}
\hspace{0.2cm}
{\raise -0.1cm \hbox{
\includegraphics[width=4cm]{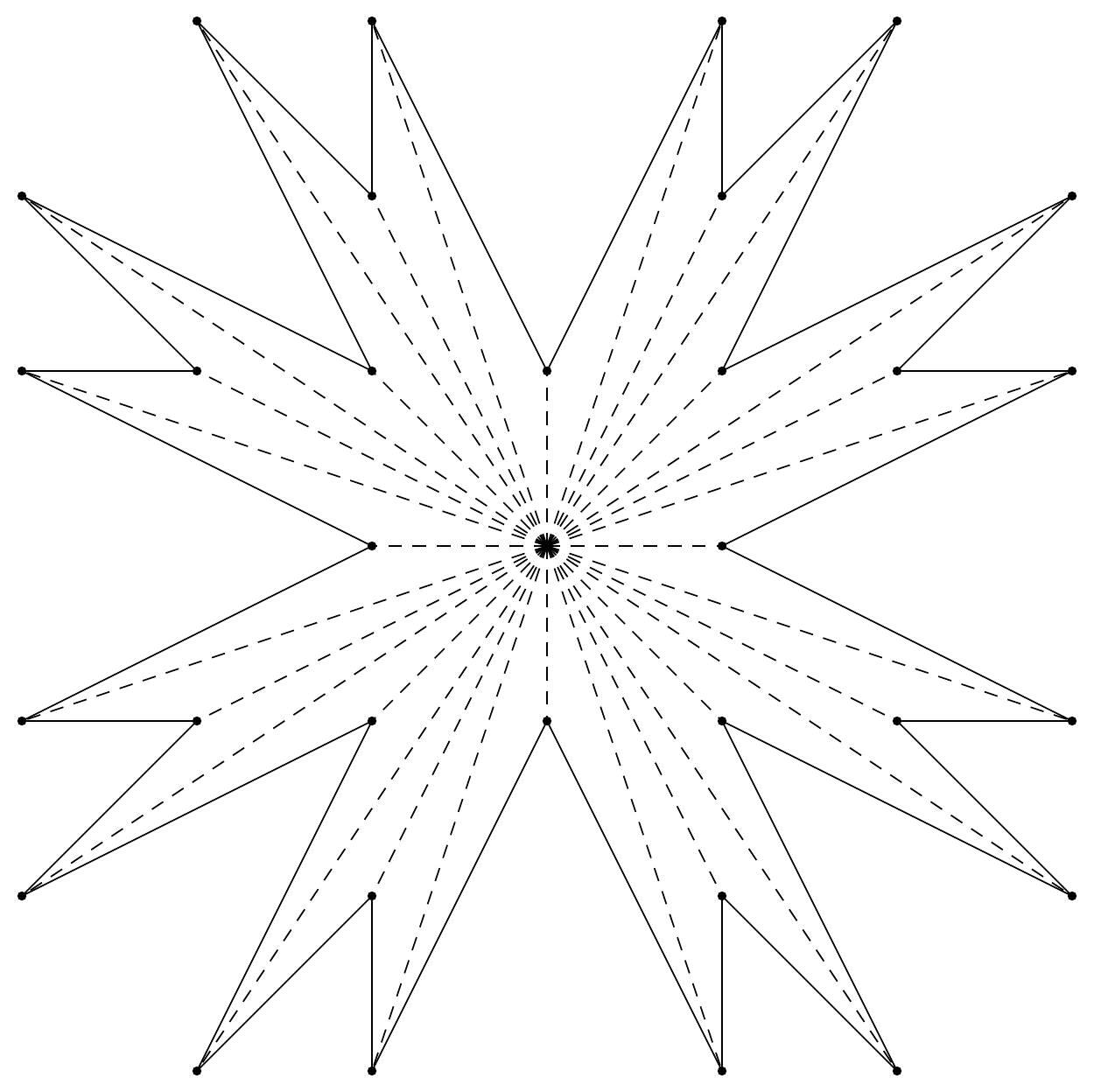}
}}
\end{center}

\caption{
\label{fig:mesh}
Three $M$-reduced meshes, produced using different strategies, for a matrix $M\in S_2^+$ of anisotropy ratio $\kappa(M)=10$ and eigenvector $(1,0.6)$. 
}
\end{figure}

For each map $\dist : \Z^d \to \R_+ \cup \{\infty\}$ and each $z\in \Z^d$, we define
\be
\label{defLambda}
\Lambda(\dist,z):= \min_{k,(\alpha_i), (v_i)} \left\{ \left\| \sum_{1 \leq i \leq k} \alpha_i v_i \right\|_M + \sum_{1 \leq i \leq k} \alpha_i \, \dist(z+v_i)\right\}, 
\ee
where the minimum is taken among all $1 \leq k \leq d$, all $\alpha_1, \cdots, \alpha_k\in \R_+$ such that $\alpha_i+ \cdots+ \alpha_k = 1$, and all non-zero vertices $v_1, \cdots,v_k$  of a common simplex $T \in \cT$. By convention, $0\times  \infty=0$. For information, in the case of a \emph{non-constant} riemannian metric $\cM$, a distinct $\cM(z)$-reduced mesh $\cT_z$ is attached to each discrete point $z$ and used to define the operator $\Lambda(\cdot,z)$, see \cite{M12}.\\

Our first result bounds the numerical error $\dist_M(z)-\distC_M(z)$, $z\in \Z^d$, in terms of the geometry of the local mesh $\cT$. We introduce the bounding radius $r_M(\cT) \in \R_+$, the radii $r_M(T)\in \R_+$ and the angles $\theta_M(T)\in [0,\pi]$, for $T \in \cT$, defined by 
\be
\label{RMT}
r_M(\cT) := \max_{T \in \cT} r_M(T), \quad r_M(T) := \max_{z\in T} \|z\|_M, \quad \cos \theta_M(T) := \min_{u,v\in T\sm\{0\}} \frac{\<u,v\>_M}{\|u\|_M \|v\|_M}.
\ee
Note that the extrema defining $r_M(T)$ and $\theta_M(T)$ are attained for vertices of $T$. We denote by $\R_+ T$ the cone spanned by a simplex $T\in \cT$ : 
$$
\R_+ T := \{r z;\, r\in \R_+, \, z\in T\}.
$$

\begin{theoremIntro}
\label{th:estim}
If the Hopf-Lax update operator $\Lambda$ is defined by \iref{defLambda}, then the fixed point problem
\be
\label{eikonal_disc_2}
\left\{
\begin{array}{lc}
\dist(z) = \Lambda(\dist,z), & z\in \Z^d \sm \{0\},\\
\dist(0) = 0
\end{array}
\right.
\ee
has a unique solution, which can be obtained ``in one pass'' using the fast marching algorithm. 
Denoting by $\dist_M : \Z^d \to \R_+$ this solution, one has for all $z\in \Z^d$
\be
\label{approx}
0 \leq \dist_M(z)-\distC_M(z) \leq d \, r_M(\cT) \left(1+ \ln^+ \left(\frac{\distC_M(z)}{r_M(\cT)}\right)\right), 
\ee
where $\ln^+(r) := \max\{\ln r,0\}$.
More precisely, if $-z \in \R_+ T$ for some $T \in \cT$, then 
\be
\label{approxT}
0 \leq \dist_M(z)-\distC_M(z) \leq d\, (\sin\theta_M(T))^2 \, r_M(T) \left(1+\ln^+\left(\frac{\distC_M(z)}{r_M(T)}\right)\right).
\ee
\end{theoremIntro}

Inequality \iref{approx} states that the discrete output $\dist_M$ of the fast marching algorithm overestimates the exact solution $\distC_M$ by a logarithmic factor, which is at most proportional to bounding radius $r_M(\cT)$. This should not be a surprise, since the bounding radius $r_M(\cT)$ reflects the discretization step as seen by the matrix $M$, and since this algorithm relies on a first order discretization of the eikonal equation.
The slightly sharper estimate \iref{approxT} takes into account the angle width of the simplex $T$. 

Our second main result is an average estimate of the radius $r_M(\cT)$, when the $M$-reduced mesh $\cT$ is constructed as described in \cite{M12}, in dimension $d \leq 3$ and as described in Proposition \ref{prop4} in dimension $4$. A key feature of this construction is that it guarantees a uniform bound on the mesh cardinality ($\#(\cT) \leq 6,\, 24$ or $768$ if $d=2, \, 3$ or $4$ respectively), and therefore a small complexity of the resulting numerical scheme, independently of the matrix $M$ and thus of its anisotropy ratio $\kappa(M)$. 

We introduce the successive Minkowski minima \cite{M12} of a matrix $M\in S_d^+$, defined for $1 \leq i \leq d$ by 
\be
\label{defMinkowski}
\lambda_i(M) := \min \{\|u_i\|_M; \, (u_1, \cdots, u_i)\in \Z^d \text{ is free, and } \|u_1\|_M \leq \cdots \leq \|u_i\|_M\}.
\ee
Note that $\lambda_1(M) \leq \cdots \leq \lambda_d(M)$ by construction. Choosing $(u_1, \cdots, u_d)$ as the canonical basis, we obtain that $\lambda_d(M) \leq \|M\|^\frac 1 2$, an upper bound which is attained in the case of a diagonal matrix. We will use Minkowski's second theorem on successive minima (\cite{M1896}, p 199), a classical result of the geometry of numbers which states that for any $M\in S_d^+$
\be
\label{Minkowski}
 \frac{2^d}{d!\omega_d} \sqrt{\det M} \leq \lambda_1(M) \cdots \lambda_d(M) \leq \frac{2^d}{\omega_d} \sqrt{\det M}, 
 \ee
where $\omega_d$ denotes the volume of the $d$-dimensional euclidean unit ball. 

It follows from Corollary 1.8 in \cite{M12} that for any $M\in S_d^+$, $1 \leq d \leq 4$, and any $M$-reduced mesh,  one has 
\be
\label{lambdaRM}
\lambda_d(M) \leq r_M(\cT).
\ee
Combining this inequality with Minkowski's second theorem on successive minima \iref{Minkowski}, one obtains the lower bound $r_d(\cT) \geq \lambda_d(M) \geq c_d (\det M)^\frac 1 {2d}$, where $c_d = 2 (d!\omega_d)^{-\frac 1 d}$ and $\omega_d$ denotes the volume of the $d$-dimensional unit ball.
The next theorem shows that this lower estimate for $r_M(\cT)$ is also an upper estimate, in an average sense, when the mesh $\cT$ is constructed with our methods.

\begin{theoremIntro}
\label{th:mu}
Let $M \in S_d^+$. If $1\leq d \leq 4$ and $\cT$ is a $M$-reduced mesh constructed as described in Proposition 1.9 or  Proposition 1.10 of \cite{M12}, or Proposition \ref{prop4} of this paper, then 
\be
\label{RLambda}
r_M(\cT) \leq K_d \lambda_d(M). 
\ee
with $K_2 = 2$, $K_3=3$, $K_4=5$.
Furthermore for any $d\geq 1$ there exists a constant $C_d$ independent of $M$ and such that 
\be
\label{avgMu}
\int_{\cO_d} \lambda_d(R^\trans M R) \,dR \leq C_d(\det M)^\frac 1 {2d}. 
\ee
We denoted by $\cO_d$ the compact group of $d \times d$ orthogonal matrices, equipped with the canonical Haar probability measure. 
\end{theoremIntro}

\begin{figure}
\begin{center}
\includegraphics[width=4.5cm]{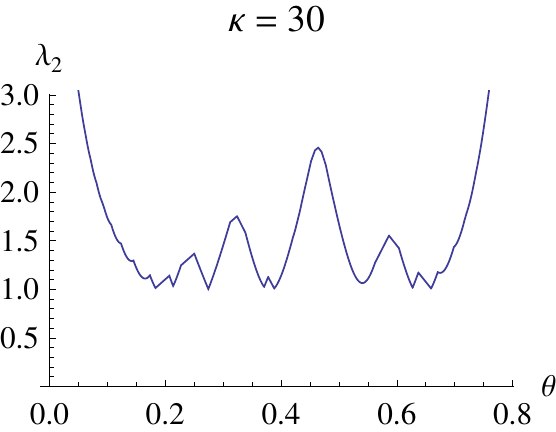}
\hspace{0.5cm}
\includegraphics[width=4.5cm]{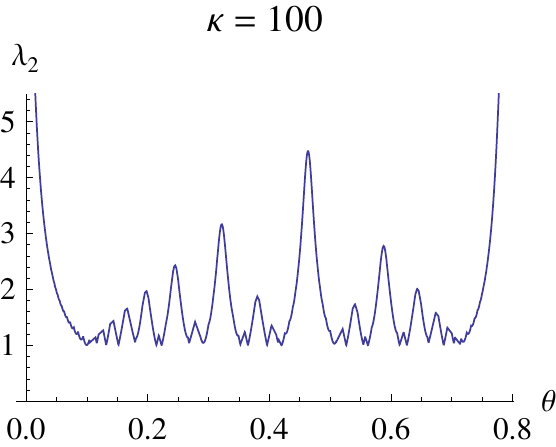}
\hspace{0.5cm}
\includegraphics[width=4.5cm]{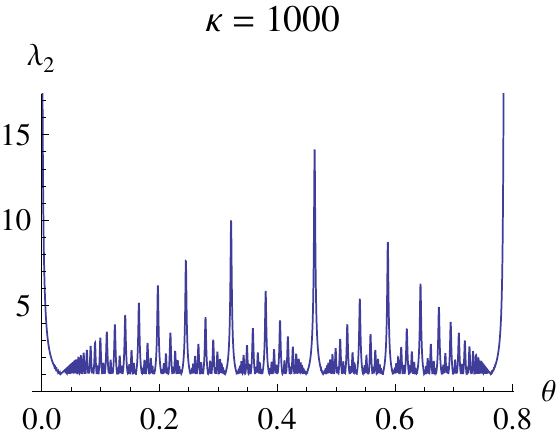}
\end{center}

\caption{
\label{fig:minkowski}
Graph of $\theta \mapsto \lambda_2(R_\theta^\trans D R_\theta)$, where $D$ is a diagonal matrix of entries $(\kappa,1/\kappa)$, and $R_\theta\in \cO_2$ a rotation of angle $\theta\in [0, \pi/4]$.
}
\end{figure}

Let us immediately stress that the scaling factor $(\det M)^\frac 1 {2d}$ does not depend on the anisotropy of $M$ : it is the same for the identity matrix and for a $2 \times 2$ matrix of eigenvalues $\lambda$, $\lambda^{-1}$. This factor only reflects the homogeneous scaling with $M$ of the Minkowski minimum $\lambda_d(M)$ (as well as the bounding radius $r_M(\cT)$ and the exact or approximated distances $\distC_M$ and $\dist_M$). In view of this natural scaling, we may therefore limit our attention to matrices satisfying $\det (M) = 1$, in which case the average upper estimate \iref{avgMu} is fully independent of $M$.
In contrast the uniform upper bound $\|M\|^\frac 1 2 \geq \lambda_d(M)$ grows like a power of the anisotropy ratio, since
$$
\kappa(M)^\frac 1 d \leq \frac {\|M\|^\frac 1 2}{(\det M)^\frac 1 {2d}} \leq \kappa(M)^\frac {d-1} d.
$$

As established in \ref{th:mu}, and illustrated on Figure \ref{fig:minkowski}, such large values of $\lambda_d(M)$ are statistically rare - in fact they correspond to pathological situations where an eigenvector associated to the small eigenvalue of $M$ is equal or close to a small element of $\Z^d\sm \{0\}$, see \S \ref{sec:mu}. 
Theorems \ref{th:estim} and \ref{th:mu} therefore yield together a genuine result of non-linear approximation, which shows that the accuracy of the algorithm introduced in \cite{M12} does not degrade as the anisotropy ratio $\kappa(M)$ increases, at least in the case of a constant metric and in an average sense over all anisotropy orientations.  

We prove Theorems \ref{th:estim} and \ref{th:mu} respectively in \S \ref{sec:acc} and \S \ref{sec:mu}. A numerical experiment is presented in \S \ref{sec:num}, as well as the extension of the algorithm to dimension four. 

\section{Error analysis of anisotropic fast marching}
\label{sec:acc}

This section is devoted to the proof of Theorem \ref{th:estim}, which is based on the notions of (discrete) sub-solution and super-solution of the fixed point discretization \iref{eikonal_disc} of the eikonal equation. We refer to \cite{BR06} for more discussions on these notions. 

We allow discrete maps to take the value $+\infty$. 

\begin{definition}
A discrete map $\dist : \Z^d \to \R_+ \cup \{\infty\}$ is a super-solution (resp.\ sub-solution) of the system \iref{eikonal_disc_2}, if it satisfies
\be
\left\{
\begin{array}{lc}
\dist(z) \geq \Lambda(\dist,z), & z\in \Z^d \sm \{0\},\\
\dist(0) = 0.
\end{array}
\right.
\qquad \left(\text{resp. }
\left\{
\begin{array}{lc}
\dist(z) \leq \Lambda(\dist,z), & z\in \Z^d \sm \{0\},\\
\dist(0) = 0.
\end{array}
\right.
\right)
\ee
\end{definition}

The exact continuous solution of the eikonal equation happens, in our specific setting, to be a discrete sub-solution, as shown in the next proposition.

\begin{prop}
\label{prop:exact}
The restriction of the exact solution $\distC_M(z) = \|z\|_M$, of the continuous eikonal equation \iref{eikonal}, to the lattice $\Z^d$, is a discrete sub-solution of the system \iref{eikonal_disc_2}.
\end{prop}

\begin{proof}
Consider $z\in \Z^d$, $1 \leq k \leq d$, non-negative coefficients $(\alpha_i)_{i=1}^k$, and vertices $(v_i)_{i=1}^k$ of a common simplex $T \in \cT$ as they appear in the Hopf-Lax operator \iref{defLambda}.
Since $\alpha_1+ \cdots+\alpha_d=1$ we obtain using first convexity and second the triangle inequality 
$$
\sum_{1 \leq i \leq k} \alpha_i \, \distC_M(z+v_i) \geq \distC_M\left(z+ \sum_{1 \leq i \leq k} \alpha_i v_i\right) \geq \distC_M(z) - \left\| \sum_{1 \leq i \leq k} \alpha_i v_i \right\|_M.
$$
It follows that $\distC_M(z) \leq \Lambda(\distC_M,z)$. Since $\distC_M(0) = 0$ this concludes the proof.
\end{proof}

We next recall an argument often referred to as the Causality property, which is at the foundation of the fast marching algorithm. See \cite{SV03} or \cite{M12} for a proof. 

\begin{lemma}[Sethian Vladimirsky, 2000, \cite{SV03}]
\label{lem:upwind}
Let $z\in \Z^d$ and let $\dist : \Z^d \to \R_+ \cup \{\infty\}$. If $\Lambda(\dist,z)<\infty$, then by compactness the minimum \iref{defLambda} defining the Hopf-Lax update $\Lambda(\dist,z)$ is attained for some $k\in \{1, \cdots, d\}$, positive coefficients $\alpha_1, \cdots, \alpha_k >0$, $\alpha_1+\cdots+\alpha_k=1$, and non-zero vertices $v_1, \cdots, v_k$ of a common simplex $T \in \cT$. We then have for all $1 \leq i \leq k$
\be
\label{upwind}
\dist(z+v_i) < \Lambda(\dist,z).
\ee
\end{lemma}

The next proposition shows that any discrete super-solution $\dist_+$ is larger than any sub-solution $\dist_-$. A similar property is proved in \cite{BR06}, on a finite discrete domain, by studying the point where the difference $(\dist_+-\dist_-)$ reaches its minimum. Since the domain $\Z^d$ is infinite we cannot rely on this approach here, and we take advantage instead of the causality property.

\begin{prop}
\label{prop:comp}
Let $\dist_+,\dist_- : \Z^d \to \R_+\cup\{\infty\}$ be respectively a discrete super-solution and sub-solution of the system \iref{eikonal_disc_2}. Then
$
\dist_- \leq \dist_+ \text{ on } \Z^d.
$
\end{prop}

\begin{proof}
We denote by $r$ (resp. $R$) the maximum (resp. minimum) radius such that the ellipsoid $\{ u\in \R^d;\,  \|u\|_M \leq r\}$ is contained in (resp. contains) the union of the elements of $\cT$, which is by assumption a compact neighborhood of the origin. 
Using first the fact that $\dist_+$ is a super-solution, and second the definition \iref{defLambda} of $\Lambda$, we obtain that for each $z\in \Z^d \sm \{0\}$ there exists a vertex $v$ of a simplex $T \in \cT$ such that 
$$
\dist_+(z) \geq \Lambda(\dist_+ ,z) \geq  r + \dist_+(z+v).
$$

Since $\dist_+$ takes non-negative values, and $\|z+v\|_M \geq \|z\|_M-R$, it immediately follows that $\dist_+(z) \geq (r/R) \|z\|_M$. Hence the set $\{ z\in \Z^d;\ \dist_+(z) \leq K\}$ is finite for any constant $K<\infty$, and it is therefore possible to sort the elements of $\{ z\in \Z^d; \, \dist_+(z) < \infty\}$ into a sequence $(z_n)_{n \geq 0}$ ordered by increasing values : $\dist_+(z_n) \leq \dist_+(z_{n+1})$ for all $n\in \Z_+$.

We prove that $d_-(z_n) \leq d_+(z_n)$ by induction on $n \in \Z_+$. We first observe that $\dist_-(z_0) = 0 = \dist_+(z_0)$, since $z_0=0$, and we next consider an arbitrary but fixed $n \in \Z_+$. In view of Lemma \ref{lem:upwind}, there exists $k\in \{1, \cdots, d\}$, some positive coefficients $(\alpha_i)_{i=1}^k$ summing up to $1$, and $n_1, \cdots, n_k\in \Z_+$ such that 
$$
\Lambda(\dist_+,z_n) = \left\| \sum_{1 \leq i \leq k} \alpha_i (z-z_{n_i}) \right\|_M + \sum_{1 \leq i \leq k} \alpha_i \, \dist_+\left(z_{n_i}\right),
$$
Lemma \ref{lem:upwind} states in addition that $d_+(z_{n_i}) < \Lambda(\dist_+,z_n) \leq d_+(z_n)$, and therefore $n_i < n$ for all $1 \leq i \leq k$. It follows that $\dist_-(z_{n_i}) \leq \dist_+(z_{n_i})$ by induction, and therefore that $\Lambda(\dist_-,z_n) \leq \Lambda(\dist_+,z_n)$. This implies $\dist_-(z_n)\leq \dist_+(z_n)$, and concludes the proof of  this proposition.
\end{proof} 

\begin{lemma}
\label{lem:decomposition}
Let $z\in \Z^d$, and let $T \in \cT$ be such that $-z\in \R_+ T$. Denoting by $v_1, \cdots, v_d$ the non-zero vertices of $T$, there exists non-negative integers $\beta_1, \cdots, \beta_d$ such that 
\be
\label{zDecomposed}
z+ \beta_1 v_1+\cdots + \beta_d v_d = 0.
\ee
\end{lemma}

\begin{proof}
Since the union of the elements of $\cT$ is a neighborhood of the origin (property (I) of the mesh $\cT$), there exists as announced a simplex $T$ such that $-z\in \R_+ T$. This inclusion implies the existence of non-negative \emph{reals} such that $z+ \beta_1 v_1+\cdots + \beta_d v_d = 0$, and we thus only have to show that $\beta_1, \cdots, \beta_d$ are integers. 
Since the vertices $v_1, \cdots, v_d$ belong to $\Z^d$, and since $|\det(v_1, \cdots, v_d)| = d! |T| = 1$ by assumption (property (II)), $(v_1, \cdots, v_d)$ is a basis of the lattice $\Z^d$. The coefficients $\beta_1, \cdots, \beta_d$ are therefore integers, which concludes the proof. 
\end{proof}

The variant studied here of the fast marching algorithm was shown in \cite{M12} to produce a solution of the system \iref{eikonal_disc} in a finite number of steps, in the case of a general continuous metric on a compact periodic domain. This algorithm involves local meshes $\cT_z$ for each discrete point $z$, which in the case of a constant metric are all equal to $\cT$. In order to adapt the proof of \cite{M12} to the case of a constant metric on an infinite domain, we briefly outline its main features.
This algorithm produces a decreasing sequence $(\dist_n)_{n \geq 0}$ of super-solutions of the system \iref{eikonal_disc_2} : $\dist_n(z) \leq \dist_m(z)$ for all $n\leq m$ and all $z\in \Z^d$. The first element $\dist_0$ of this sequence satisfies $\dist_0(0) = 0$ and $\dist_0(z) = \infty$ for all $z\neq 0$.
Simultaneously, this algorithm produces an injective sequence of points $(z_n)_{n \geq 0}$ which has the following properties :

(i)\phantom{ii} $\dist_n(z_n)$ is the minimum value of $\dist_n$ on the set $\Z^d \sm \{z_m;\, m<n\}$.

(ii)\phantom{i} $\dist_n(z_n) = d_m(z_n) = \Lambda(\dist_m,z_n)$ for all $m\geq n>0$.

(iii) $\dist_{n+1}$ takes finite values on the set 
$\{z\in \Z^d;\, z_n = z+v \text{ for some vertex $v$ of } \cT_z\}$, 
and coincides with $\dist_n$ outside of this set. 
\begin{prop}
\label{prop:solve}
The set $Z := \{z_n; \, n \geq 0\}$ coincides with $\Z^d$, and a solution $\dist$ to the system \iref{eikonal_disc_2} is given by the decreasing limit, for all $z\in \Z^d$, 
\be
\label{lim_dist}
\dist(z) := \lim_{n \to \infty} \dist_n(z).
\ee
\end{prop}

\begin{proof}
We first observe that $\dist_n(z) \geq \|z\|_M$ for all $z\in \Z^d$ and all $n \geq 0$, since $\dist_n$ is a super-solution  of the system \iref{eikonal_disc_2}, and  $z\mapsto \|z\|_M$ a sub-solution. 

Let $n \geq 0$ be arbitrary, and let $v$ be an arbitrary vertex of $\cT$. Since $d_{n+1}(z_n-v) < \infty$ (using property (iii) of the sequence $(z_n)_{n \geq 0}$, and recalling that $\cT=\cT_z$ for a constant metric), the cardinality $n' := \# \{ x \in \Z^d; \, \|x\|_M \leq d_{n+1}(z_n-v)\}$ is finite, which shows that $z_n-v = z_m$ for some $m<n'$ (using (i)).

As a result for any $z\in Z$ and any vertex $v$ of $\cT$ we have $z-v\in Z$. Since the minimum value of $d_0$ is attained at the origin only, we have $z_0=0$ and therefore $0\in Z$. Since any $z\in \Z^d$ can be written under the form \iref{zDecomposed}, it follows that $Z = \Z^d$. In view of (ii), the discrete map defined by \iref{lim_dist} is a solution of the system \iref{eikonal_disc_2}, which concludes the proof.
\end{proof}

The fast marching algorithm is regarded as a ``one pass algorithm'' because the sequence $(z_n)_{n \geq 0}$ is injective. In the case of a constant metric this implies in particular, using property (iii), that for any $z\in \Z^d$, the cardinality of $\{ n \geq 0;\, d_{n+1}(z) \neq d_n(z)\}$ is bounded by the number of vertices of $\cT$. Proposition  \ref{prop:solve} establishes the first part of Theorem \ref{th:estim}, while Propositions \ref{prop:exact} and \ref{prop:comp} imply the lower estimates for $\dist_M$ in \iref{approx} and \iref{approxT}.

\subsection{Construction of a discrete super-solution}

We construct in this section an explicit discrete super-solution $\dist_+$ of the system \iref{eikonal_disc_2}. Using its expression, and the fact that super-solutions are larger than solutions, we obtain the upper estimates for $\dist_M$ announced in Theorem \ref{th:estim}, \iref{approx} and \iref{approxT}, which concludes its proof.

We consider an arbitrary simplex $T \in \cT$, which is fixed throughout this section, and 
we define a discrete map $\dist_+ : \Z^d \to \R_+ \cup \{\infty\}$ as follows. If $z\in \Z^d$ is such that $-z \notin \R_+T$, then we set $\dist_+(z):=\infty$. Otherwise we write $z$ under the form \iref{zDecomposed}, and we define
\be
\label{defd+}
\dist_+(z) := \|z\|_M + (\sin \theta_M(T))^2 \sum_{1 \leq i\leq d} s(\beta_i) \|v_i\|_M  , \quad \text{ where } s(\beta) :=  \sum_{1 \leq k \leq \beta} \frac 1 k. 
\ee
The summation appearing in $\dist_+$ can be bounded as follows :
\be
\label{ineqSum}
\sum_{1 \leq i\leq d} s(\beta_i) \|v_i\|_M  \leq \sum_{1 \leq i\leq d}  \|v_i\|_M \left(1+\ln^+\left(\frac{\|z\|_M}{\|v_i\|_M}\right)\right) \leq \, d \, r_M(T) \left(1+\ln^+\left(\frac{\|z\|_M}{r_M(T)}\right)\right)
\ee
where we used in the first inequality the upper bound $\beta_i \leq \|z\|_M/\|v_i\|_M$ (due to the acuteness condition \iref{acuteness}), and $\sum_{k=1}^\beta k^{-1} \leq 1+\ln^+(\beta)$. For the second inequality we used the growth of $r \mapsto r \ln (1+\ln^+ (\lambda/r))$ on $\R_+^*$, for any fixed $\lambda>0$. Let us assume for a moment that $\dist_+$ is a super-solution, which implies that $\dist_+ \geq \dist_M$. Then combining \iref{ineqSum} and \iref{defd+} we obtain the upper estimates in \iref{approxT}. Since the simplex $T\in \cT$ is arbitrary, we also obtain \iref{approx}, which concludes the proof of Theorem \ref{th:estim}. 

In order to prove that $\dist_+$ is a super-solution, we need a preliminary lemma, on the Taylor development of the $\|\cdot \|_M$ norm, under an acuteness condition. 

\begin{lemma}
\label{lem:Taylor}
Let $z,v\in \R^d\sm\{0\}$ and $\theta_0\in [0,\pi/2]$ be such that 
\be
\label{angleUE}
\<v,z-v\>_M \geq \cos(\theta_0) \|v\|_M \|z-v\|_M.
\ee
Then
$$
\|z-v\|_M \leq \|z\|_M - \frac{\<z,v\>_M}{\|z\|_M}+\frac{(\sin \theta_0)^2\|v\|_M^2}{\|z\|_M}.
$$
\end{lemma}

\begin{proof}
Since the vectors $v$ and $z-v$ form an angle smaller than $\theta_0$ \iref{angleUE}, this is also the case for the vectors $v$ and $z-t v = z-v + (1-t) v$, for $t\in [0,1]$. Therefore $\<v,z-t v\>_M \geq \cos(\theta_0) \|v\|_M \|z-t v\|_M$, for $t\in [0,1]$.
On the other hand we have $\|z\|_M \|v\|_M \geq \<z,v\>_M \geq \<z-v,v\>_M+ \|v\|_M^2 \geq \|v\|_M^2$, hence $\|z\|_M \geq \|v\|_M$ which implies that $\|z- t v\|_M \geq \|z\|_M-t \|v\|_M \geq (1-t)\|z\|_M$, for $t\in [0,1]$.

The Taylor formula with integral rest, applied to the smooth function $t \mapsto \|z-t v\|_M$, yields
$$
\|z-v\|_M = \|z\|_M - \frac{\<z,v\>_M}{\|z\|_M} + \int_0^1 \left( \frac{\|v\|_M^2}{\|z-t v\|_M}  - \frac {\<z-t v,v\>_M^2}{\|z-t v\|_M^3} \right) (1-t) dt.
$$
Injecting the lower bounds $\<z - t v,v\>_M \geq \cos(\theta_0) \|v\|_M \|z-t v\|_M$, and $\|z-t v\|_M\geq (1-t) \|z\|_M$, we obtain that the integrated term is bounded above by $(\sin \theta_0)^2 \|v\|_M^2/\|z\|_M$. This concludes the proof of this lemma.
\end{proof}

\begin{prop}
\label{prop:infinite_needed}
The map $\dist_+$ defined by \iref{defd+} is a discrete super-solution of \iref{eikonal_disc_2}.
\end{prop}

\begin{proof}
As required, we have $\dist_+(0) = 0$. We thus consider an arbitrary $z\in \Z^d\sm\{0\}$, and we show in the following that $\dist_+(z) \geq \Lambda(\dist_+,z)$. There is nothing to prove if $\dist_+(z) = +\infty$, and we may therefore write $z$ under the form \iref{zDecomposed} : $z+\sum_{i=1}^d \beta_i v_i=0$, where $\beta_i\in \Z_+$ for $1 \leq i \leq d$ and $v_1, \cdots, v_d$ are the vertices of $T$. 

For all $1 \leq i \leq d$ such that $\beta_i >0$, we apply Lemma \ref{lem:Taylor} to $z$, $v:=-v_i$, and the angle $\theta_M(T)$. We obtain, since $\|z\|_M \geq \beta_i \|v_i\|_M$ (due to the acuteness condition \iref{acuteness}),
$$
\|z\|_M - \|z+v_i\|_M \geq \frac{\<-z,v_i\>_M}{\|z\|_M} - (\sin \theta_M(T))^2 \frac{\|v_i\|_M^2}{\|z\|_M} 
 \geq \frac{\<-z,v_i\>_M}{\|z\|_M} - (\sin \theta_M(T))^2 \frac{\|v_i\|_M}{\beta_i}.
$$
Hence $\dist_+(z)-\dist_+(z+v_i) \geq \<-z,v_i\>_M /\|z\|_M$, for all $1 \leq i \leq d$ such that $\beta_i>0$.
Defining $\alpha_i := \beta_i/(\beta_1+\cdots+\beta_d)$, in such way that $\sum_{i=1}^d \alpha_i v_i$ is positively proportional to $-z$, we obtain
$$
\left\| \sum_{1 \leq i \leq d} \alpha_i v_i \right\|_M + \sum_{1 \leq i \leq d} \alpha_i \, \dist_+(z+v_i) \leq 
\left\| \sum_{1 \leq i \leq d} \alpha_i v_i \right\|_M + \dist_+(z) + \left\<\sum_{1 \leq i \leq k} \alpha_i v_i, \, \frac z {\|z\|_M}\right\>_M = 
\dist_+(z).
$$
This establishes that $\dist_+(z) \geq \Lambda(\dist_+,z)$, which concludes the proof.
\end{proof}

\section{Average value of the outer radius $r_M(\cT)$}
\label{sec:mu}

This section is devoted to the proof of Theorem \ref{th:mu}, which begins with the upper bound \iref{RLambda} on $r_M(\cT)$. An inspection of the construction of the mesh $\cT$ in Proposition 1.9 or  Proposition 1.10 of \cite{M12}, in dimension $d=2$ or $3$ respectively, shows that its vertices have the form respectively
$$
\ve_1 u_1+ \ve_2 u_2, \ \text{ or } \ \ve_1 u_1+ \ve_2 u_2 + \ve_3 u_3,
$$
where $\|u_i\|_M = \lambda_i(M)$ and $\ve_i\in \{-1,0,1\}$, for all $1 \leq i \leq d$. The upper bound on $r_M(\cT)$ thus immediately follows from the triangle inequality and $\lambda_1(M) \leq \cdots \leq \lambda_d(M)$. In the four dimensional case, which is covered by Proposition \ref{prop4} of this paper, the vertices also have the form $\ve_1 u_1+ \ve_2 u_2 + \ve_3 u_3+\ve_4 u_4$, where three of the $|\ve_i|$ are bounded by $1$ and the remaining one is bounded by $2$. The announced inequality \iref{RLambda} again follows from the triangle inequality.\\


We next turn to the proof of \iref{avgMu}, in which the dimension $d\geq 1$ is arbitrary.
The orthogonal group $\cO_d$ is equipped with the canonical Haar probability measure. 
We denote by $R$ a random orthogonal matrix, and by $\cP$ the probability of an event in $\cO_d$.

For each matrix $M\in S_d^+$, we denote by $0 < \nu_1(M) \leq \cdots \leq \nu_d(M)$ the square roots of the successive eigenvalues of $M$. In particular $\nu_1(M) = \|M^{-1}\|^{-\frac 1 2}$, and $\nu_d(M) = \|M\|^\frac 1 2$. We denote by $|u| := \sqrt{u^\trans u}$ the \emph{euclidean} norm of a vector $u\in \R^d$.

\begin{lemma}
\label{lem:u}
Let $M\in S_d^+$ be such that $\det(M)=1$. 
Then for any $u\in \R^d$ such that $|u|=1$ one has 
\be
\label{proba:u}
\cP(\, \|R u\|_M \leq \delta\, ) \leq C(d) \, \delta^{d-1} \, \nu_1(M), 
\ee
where $C(d) = 2^d/ \omega_d$, and $\omega_d$ is the volume of the $d$-dimensional euclidean unit ball.
\end{lemma}

\begin{proof}
We denote by $B := \{x\in \R^d; \, |x| \leq 1\}$ the euclidean unit ball, and we introduce the set 
$$
B_0 := \{x\in B; \|x\|_M \leq \delta |x|\}.
$$
Since the Haar measure on $\cO_d$ is invariant under the action of rotations, its image by $R \mapsto R u$ is the uniform probability on the euclidean sphere. 
The probability estimated in \iref{proba:u} is therefore equal to $|B_0|/|B|$. 
The invariance of the Haar measure under rotations also allows us to assume without loss of generality that $M$ is a diagonal matrix. For notational simplicity we denote $\nu_i := \nu_i(M)$, for $1\leq i \leq d$, and we observe that  $\nu_1 \cdots \nu_d = \sqrt{\det M} = 1$.

If $x = (x_1, \cdots, x_d)\in B_0$, then 
$$
 \nu_1^2 x_1^2 + \cdots + \nu_d^2 x_d^2 \leq \delta^2 |x|^2 \leq \delta^2,
$$
hence
$$
x\in [-1,1] \times [-\delta/\nu_2, \delta/\nu_2] \times \cdots \times [-\delta/\nu_d, \delta/\nu_d].
$$
It follows that 
$
|B_0| \leq 2^d \delta^{d-1}  /(\nu_2 \cdots \nu_d) = 2^d \delta^{d-1} \nu_1,
$
which concludes the proof.
\end{proof}

Our next result is an estimate in probability of the first Minkowski minimum of $M$ : 
$$
\lambda_1(M) := \min_{u\in \Z^d\sm\{0\}} \|u\|_M.
$$
\begin{corollary}
\label{corol:lambda1}
Let $M \in S_d^+$ be such that $\det(M) = 1$. Then for each $\delta>0$ 
\be
\label{proba:min}
\cP(\,\lambda_1(R^\trans M R) \leq \delta\,) \leq C'(d) \, \delta^d,
\ee
where $C'(d) := 2 d 3 ^{d-1} C(d)$, and $C(d)$ denotes the constant from Lemma \ref{lem:u}.
\end{corollary}

\begin{proof}
We introduce the collection of points $E := \{ u\in \Z^d ;\, 0 < |u| \leq \delta/\nu_1(M)\}$, and we observe that $\|R u\|_M \geq \delta$ for any $u \in \Z^d \sm (E \cup \{0\})$ and any $R \in \cO_d$. 
We obtain
\begin{eqnarray}
\nonumber
\cP(\, \lambda_1(R^\trans M R) \leq \delta\, ) &\leq& \sum_{u\in \Z^d\sm\{0\}} \cP( \, \|R u\|_M \leq \delta\,)\\
\nonumber
&=& \sum_{u\in E} \cP( \, \|R u\|_M \leq \delta\,)\\
\label{PLambdaSum}
& \leq & C(d) \, \nu_1(M) \sum_{u \in E} (\delta/|u|)^{d-1},
\end{eqnarray}
where in the third line we applied Lemma \ref{lem:u} to the normalized vectors $u/|u|$, $u\in E$.

In order to estimate $\sum_{u \in E} (\delta/|u|)^{d-1}$, we introduce the $\|\cdot \|_\infty$ norm on $\R^d$ defined by $\|(u_1, \cdots, u_d)\|_\infty := \max\{|u_1|, \cdots, |u_d|\}$. We observe that $\|u\|_\infty \leq |u|$, and that for each positive integer $k$ there exists exactly $(2k+1)^d-(2k-1)^d$ elements $u\in \Z^d$ which satisfy $\|u\|_\infty = k$.
Therefore 
\be
\label{sumE}
\sum_{u\in E} |u|^{1-d} \leq \sum_{u\in E} \|u\|_\infty^{1-d} = \sum_{0 < k \leq \delta /\nu_1(M)} \frac{(2k+1)^d-(2k-1)^d}{k^{d-1}}
\ee
Remarking that 
$(2k+1)^d-(2k-1)^d \leq d(2k+1)^{d-1} ((2k+1)-(2k-1)) \leq 2 d (3k)^{d-1}$, 
we bound the right hand side of \iref{sumE} by $(\delta /\nu_1(M)) 2 d 3^{d-1}$. Combining this estimate with \iref{PLambdaSum}, we conclude the proof. 
\end{proof}

The right part of Minkowski's second theorem on successive minima \iref{Minkowski} implies
$$
\lambda_d(M) \leq \frac{2^d}{\omega_d} \frac{\sqrt{\det M}}{\lambda_1(M)^{d-1}}.
$$
For any matrix $M\in S_d^+$ such that $\det(M)=1$, we thus obtain
\begin{eqnarray*}
\frac{\omega_d}{2^d}
\int_{\cO_d} \lambda_d(R^\trans M R) dR 
& \leq & \int_{\cO_d} \frac{dR}{\lambda_1(R^\trans M R)^{d-1}} \\
& = & (d-1) \int_{0}^\infty \cP(\, \lambda_1(R^\trans M R)\leq \delta\, ) \frac{d \delta}{\delta^d}\\
& \leq & (d-1) \int_0^1 C'(d) \delta^d \, \frac{d \delta}{\delta^d} + (d-1) \int_1^\infty \frac{d \delta}{\delta^d}\\
&=&(d-1)C'(d)+1. 
\end{eqnarray*}
This concludes the proof of \iref{avgMu} in the case of a matrix $M$ of determinant $1$. 

The general case of an arbitrary determinant can be reduced to the case $\det(M)=1$ by homogeneity. Indeed the determinant is homogeneous : $\det (\mu^2 M) = \mu^{2d} \det (M)$, for any $\mu>0$; and so are Minkowski's minima : $\lambda_i(\mu^2 M) = \mu \,\lambda_i(M)$ (this is an immediate consequence of the homogeneity of the norm : $\|u\|_{\mu^2 M} = \mu \|u\|_M$ for any $u\in \R^d$). 


\section{Numerical experiments, and extension to dimension $4$}
\label{sec:num}
We discuss in the first part of this section some two-dimensional numerical experiments, which illustrate our two main results Theorem \ref{th:estim} and Theorem \ref{th:mu}. These numerical experiments are conducted in the restrictive setting which is the framework of these results: a constant riemannian metric. The reader interested in benchmarks closer to applications is referred to \cite{M12}.
The second part of this section is devoted to the extension of our variant of the fast marching algorithm to four dimensional domains.\\

We conducted some numerical experiments, in which the anisotropic eikonal equation associated to a constant riemannian metric $\cM = M\in S_2^+$ is discretized on the finite grid $\Omega_* := \{-500, \cdots, 500\}^2 \sm \{(0,0)\}$ and solved using three different methods. The matrix $M$ has eigenvalues $1/10$ and $10$, thus anisotropy ratio $\kappa(M) = 10$, and the eigenvector associated to the eigenvalue $1/10$ is $(1, 0.6)$. The fast marching algorithm (FM) is applied with two different $M$-reduced meshes $\cT_1$ and $\cT_2$ presented on Figure \ref{fig:dependency} (top and bottom respectively), of cardinality 6 and 10 respectively.
The mesh $\cT_1$ is constructed as described in Proposition 1.9 in \cite{M12}. The mesh $\cT_2$ is obtained using a different strategy,  which will be the object of a future article : an iterative subdivision procedure which stops when the mesh satisfies the acuteness condition \iref{acuteness}. 
The results are compared with with the Iterative Iterative Gauss Siedel (IGS) algorithm introduced in \cite{BR06}.
$$
\begin{array}{c|c|c|c|} 
 & \text{IGS} & \text{FM }\cT_1 & \text{FM }\cT_2 \\ 
 \hline
\|\distC-\dist\|_{L^\infty(\Omega_*)} & 25.2 & 2.69 & 0.803 \\ 
\|\distC-\dist\|_{L^\infty(\Omega^1_*)} & 25.2 & 1.25 & 0.803 \\ 
\text{Timing} & 2.0s & 3.6s & 13.1s \\ 
\end{array}
$$
The computation times were obtained on a 2.4Ghz core 2 duo laptop. Surprisingly, the IGS is the fastest algorithm on this test case, but it is also the most imprecise.  Indeed the numerical error (on the first line of the table) is reduced by a factor $9$ when one replaces the IGS with the fast marching algorithm with the mesh $\cT_1$, and by an additional factor $3$ with the mesh $\cT_2$. 
For comparison, the original \emph{isotropic} fast marching algorithm, with $M' = \Id$ and on the same domain, yields a numerical error of $2.1$, which confirms that our variant of the fast marching algorithm did not suffer from the anisotropy of $M$, contrary to the IGS.

\begin{figure}
\begin{center}
\includegraphics[width=5cm]{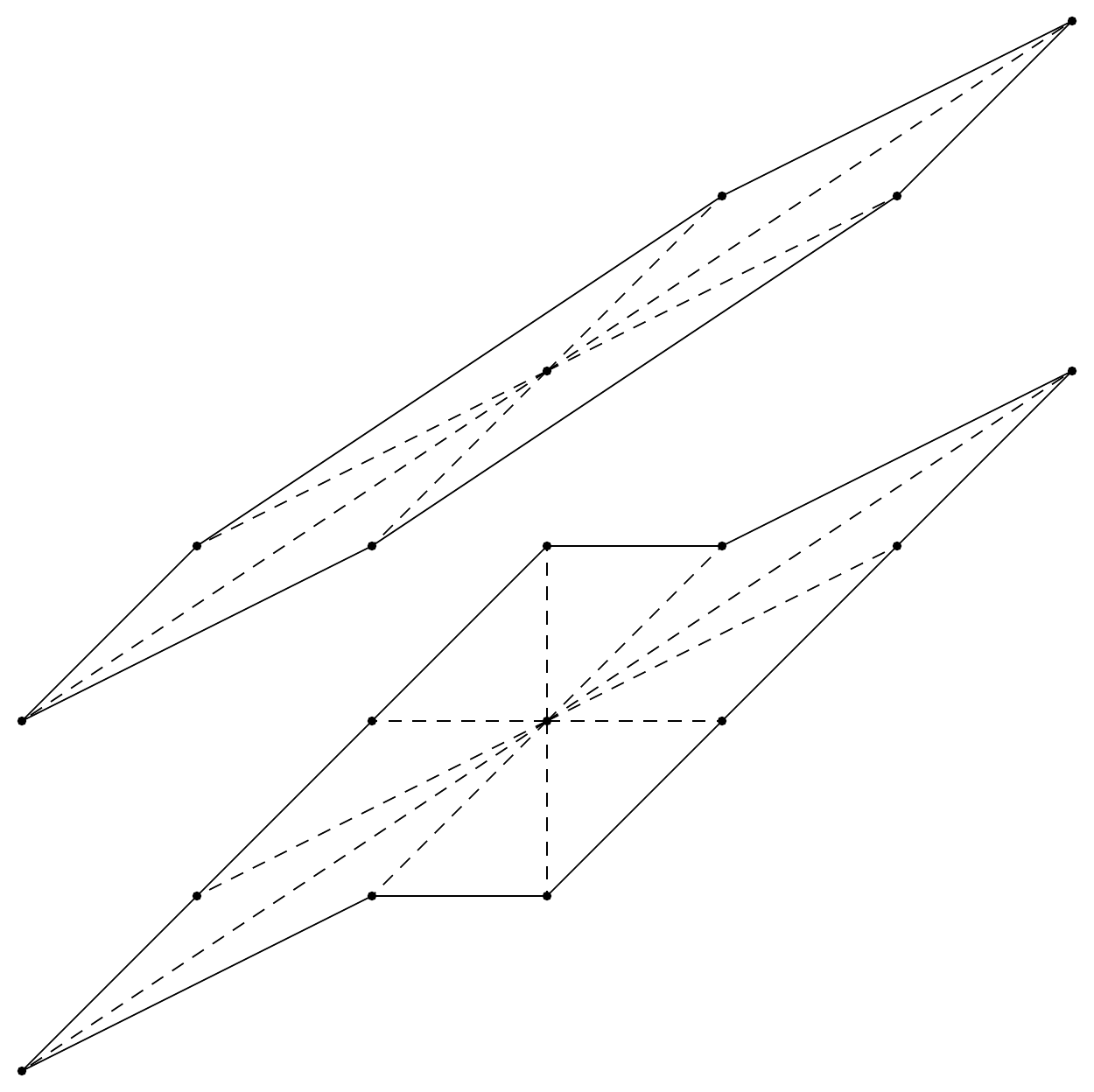}
\includegraphics[width=5cm]{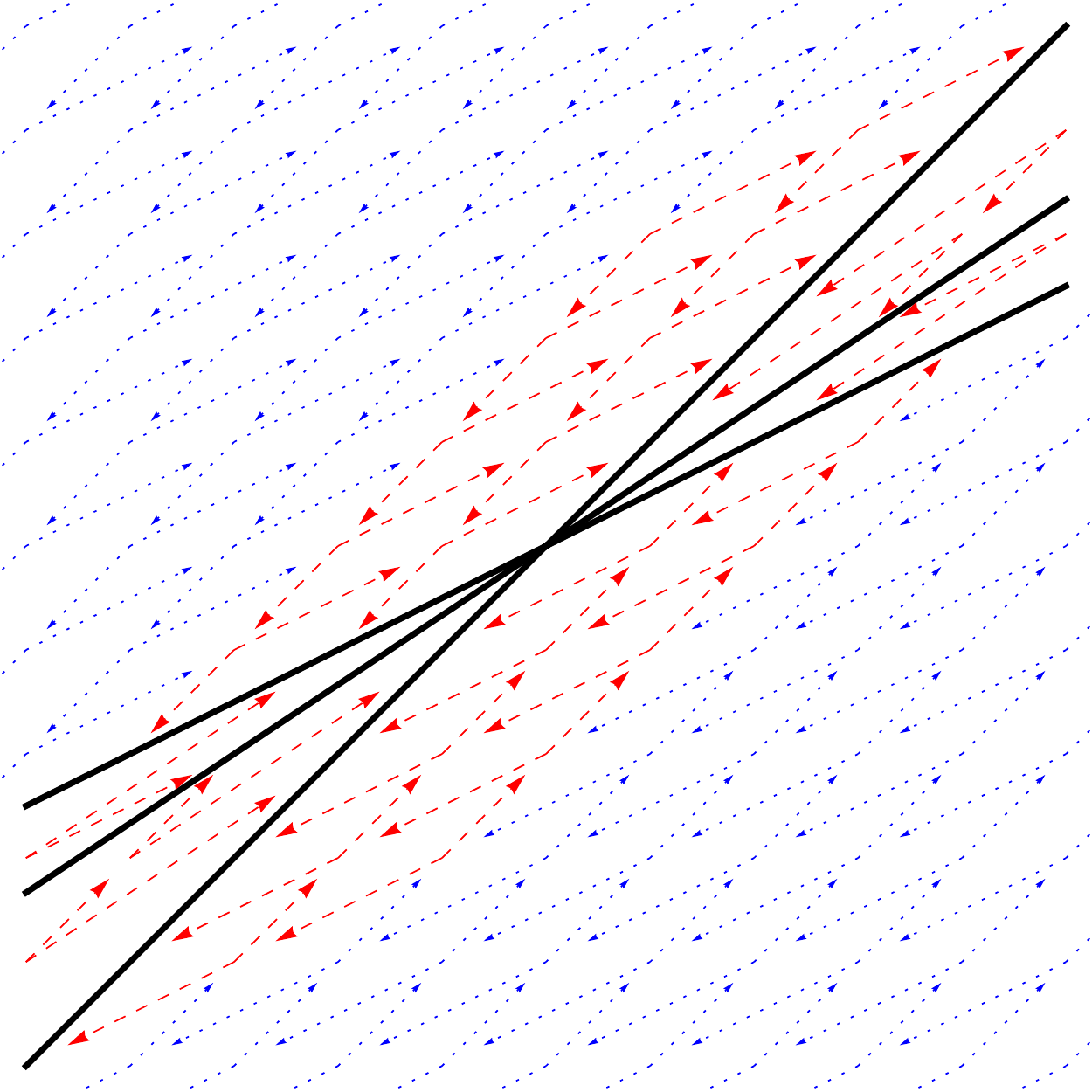}
\includegraphics[width=5cm]{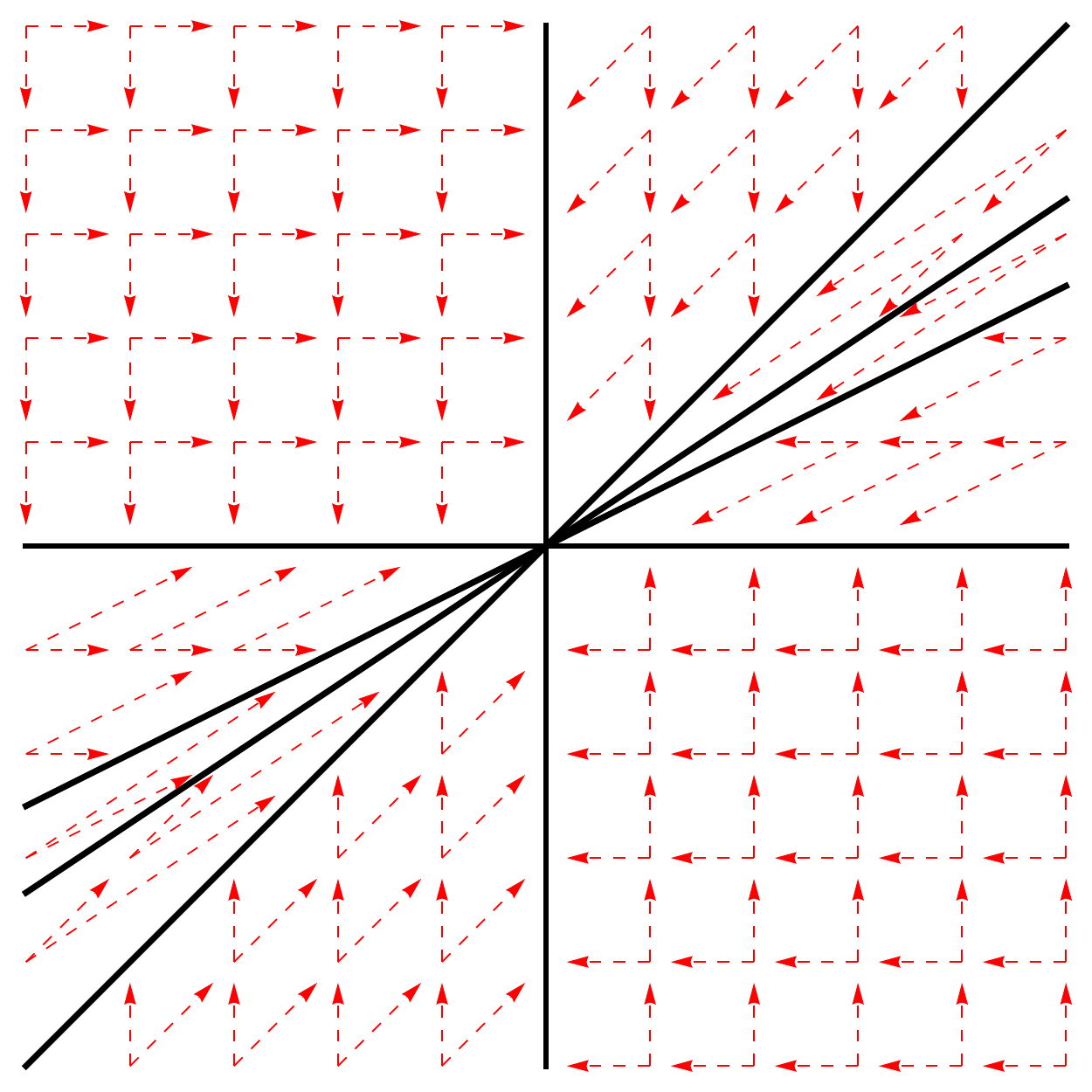}
\end{center}
\label{fig:dependency}
\caption{The two meshes $\cT_1$ (top) and $\cT_2$ (bottom). Portion of the domain where our error analysis of the fast marching algorithm applies (thick, dashed, red arrows) for $\cT_1$ (center) and $\cT_2$ (right).} 
\end{figure}

The numerical error is defined as the maximum absolute value of the difference $\distC-\dist$ between the exact solution of the eikonal equation $\distC(z) = \|z\|_M$, and a discrete approximation $\dist$ produced by the algorithm of interest. The maximum is taken either on the whole finite grid $\Omega_*$ for the first line of the table, or for the second line on a subset $\Omega_*^1$ which is described in the next remark.

\begin{remark}
The error analysis presented in Theorem \ref{th:estim} was obtained on the infinite grid $\Z^2\sm \{(0,0)\}$, instead of the finite grid $\Omega_* =\{-500,\cdots,500\}^2\sm \{(0,0)\}$, and is not automatically valid on the whole $\Omega_*$ for the following reason. Consider two non-zero vertices $u,v$ of a triangle $T$ in the $M$-reduced mesh $\cT$, and $\alpha, \beta \in \Z_+$ such that $z:=-(\alpha u + \beta v)$ belongs to the discrete domain. We used in Proposition \ref{prop:infinite_needed} that $-(\alpha' u+\beta' v)$ also belongs to the discrete domain, for all $0 \leq \alpha' \leq \alpha$ and all $0 \leq \beta' \leq \beta$. 

This property does not hold for the mesh $\cT_1$ and the finite grid $\Omega_*$, but does hold for the mesh $\cT_2$ on the same grid, see Figure \ref{fig:dependency} (center and right respectively). More precisely, for the mesh $\cT_1$, this property holds on a subset $\Omega_*^1$ which contains approximately $37\%$ of the elements of $\Omega_*$, and is illustrated on Figure \ref{fig:dependency} (center; thick, dashed, red arrows). As expected, and illustrated on the second line of the table, the numerical error of the fast marching algorithm with the $M$-reduced mesh $\cT_1$ is significantly reduced on this subset, while it is unchanged with the mesh $\cT_2$ or the IGS.\\
\end{remark}

This rest of this section is devoted to the proof of the following proposition, which extends the variant of the fast marching algorithm presented in \cite{M12} to four dimensional domains. 

A $M$-reduced basis, where $M\in S_d^+$, $1 \leq d \leq 4$, is a collection $(u_1, \cdots, u_d)$ of elements of $\Z^d$ satisfying $|\det(u_1, \cdots, u_d)| = 1$ and $\|u_i\|_M = \lambda_i(M)$. Such a basis exists and can be obtained with an algorithm of complexity $\cO(\ln \kappa(M))$, if one regards the elementary operations among reals $(+,-,\times,/)$ as of unit complexity, see \cite{NS09,M12}.

\begin{prop}
\label{prop4}
Let $M\in S_4^+$ and let $(u_1, u_2, u_3, u_4)$ be a $M$-reduced basis. Consider the collection $\cT$ of simplices which have the following vertices
\begin{eqnarray}
\label{simpl1}
&(0,\, v_1, \ v_1+v_2, \ v_1+v_2+v_3, \ 2v_1+v_2+v_3+ v_4), &\\
\label{simpl2}
&(0, \, v_1+v_2, \ v_1+v_2+v_3, \ v_1+v_2+v_3+v_4, \ 2v_1+v_2+v_3+ v_4),&
\end{eqnarray}
where $(v_1,v_2,v_3,v_4)$ runs over all permutations of $(\ve_1 u_1, \, \ve_2 u_2, \, \ve_3 u_3, \, \ve_4 u_4)$, with arbitrary signs $(\ve_1, \, \ve_2,\, \ve_3,\, \ve_4)\in \{-1,1\}^4$. Then $\cT$ is a $M$-reduced mesh, of cardinality $768$.
\end{prop}

The cardinality of $\#(\cT)$ is indeed $2\times 4! \times 2^4=768$, where the factor $2$ stands for the two types of simplices, $4!$ is the number of permutations of a four elements set, and $2^4$ is the number of choices for the signs $(\ve_1, \, \ve_2,\, \ve_3,\, \ve_4)$. 
The fact that $\cT$ is a conforming mesh, and that the union of its elements is a neighborhood of the origin, is easily checked when $(u_1, u_2, u_3,u_4)$ is the canonical basis of $\R^d$. It thus holds for an arbitrary basis $(u_1, u_2, u_3,u_4)$ of $\R^d$ by change of variables. Observing that 
\begin{eqnarray*}
|\det(v_1, \ v_1+v_2, \ v_1+v_2+v_3, \ 2v_1+v_2+v_3+ v_4)| &=& |\det(u_1, \, u_2, \, u_3, \, u_4)|,\\
|\det (v_1+v_2, \ v_1+v_2+v_3, \ v_1+v_2+v_3+v_4, \ 2v_1+v_2+v_3+ v_4)| &=& |\det(u_1, \, u_2, \, u_3, \, u_4)|,
\end{eqnarray*}
and recalling that $|\det(u_1, u_2, u_3,  u_4)| = 1$, by definition, we obtain that the volume of each element of $\cT$ is $1/4!$. In order to conclude the proof, it only remains to establish the acuteness property \iref{acuteness}.



We now recall a property of the $M$-reduced basis $(u_1, u_2, u_3, u_4)$, established in Proposition 1.6 of \cite{M12}.
For any integer combination $z$ of the elements of the basis distinct from $u_i$, in other words $z=\alpha_1 u_1+ \cdots +\alpha_{i-1} u_{i-1}+ \alpha_{i+1} u_{i+1}+ \cdots+\alpha_4 u_4$, where $\alpha_1, \cdots, \alpha_{i-1}, \alpha_{i+1}, \cdots, \alpha_4\in \Z$, one has 
\be
\label{scalNormIneq}
2 |\<z,u_i\>_M| \leq \|z\|_M^2.
\ee

The following lines establish that the angle $\<u,v\>_M$ formed by any pair of non-zero vertices $u,v$ of any simplex $T \in \cT$ is non-negative. For that purpose we distribute by bi-linearity the terms of the scalar product $\<u,v\>_M$, and we group them within square brackets. It easily follows from \iref{scalNormIneq} that the sum of the contents of each square bracket is non-negative. 
In the following three cases, this inequality should be applied with $z=v_1$.
\begin{eqnarray*}
\<v_1,\, v_1+v_2\>_M & = & [\,  \|v_1\|_M^2 + \<v_1, v_2\>_M \, ].\\
\<v_1,\, v_1+v_2+v_3\>_M & = & [\,  \|v_1\|_M^2 + \<v_1,v_2\>_M + \<v_1, v_3\>_M\, ].\\
\<v_1,\, 2 v_1+v_2+v_3+v_4\>_M & = & [\,  2\|v_1\|_M^2 + \<v_1, v_2\>_M + \<v_1, v_3\>_M+\<v_1,v_4\>_M\, ].
\end{eqnarray*}
In the next three cases, one should apply \iref{scalNormIneq} with $z=v_1$, $z=v_1+v_2$ or $z=v_1+v_2+v_3$.
\begin{eqnarray*}
\<v_1+v_2,\, v_1+v_2+v_3\>_M &=& [ \, \|v_1+v_2\|_M^2+\<v_1+v_2,v_3\>_M \,].\\
\<v_1+v_2,\, 2v_1+v_2+v_3+v_4\>_M &=& [ \, \|v_1+v_2\|_M^2+\<v_1+v_2,v_3\>_M + \<v_1+v_2,v_4\>_M \,]\\
& & + [\, \|v_1\|_M^2+\<v_1,v_2\>_M\, ].\\
\<v_1+v_2+v_3,\, 2v_1+v_2+v_3+v_4\>_M &=& [ \, \|v_1+v_2+v_3\|_M^2+\<v_1+v_2+v_3,v_4\>_M \,]\\
& & + [\, \|v_1\|_M^2+\<v_1,v_2\>_M+\<v_1,v_3\>_M\, ].
\end{eqnarray*}
The non-negativity of the above expressions shows as announced that any two non-zero vertices $u,v$ of any simplex of the form \iref{simpl1} form a non-negative scalar product. We obtain a similar conclusion for simplices of the form \iref{simpl2} by checking that the following scalar products are non-negative. One should apply \iref{scalNormIneq} with $z=v_1$, $z=v_1+v_2$, $z=v_1+v_2+v_3$ or $z=v_2+v_3+v_4$.
\begin{eqnarray*}
\<v_1+v_2,\, v_1+v_2+v_3+v_4\>_M &=& [ \, \|v_1+v_2\|_M^2+\<v_1+v_2,v_3\>_M + \<v_1+v_2,v_4\>_M \,].\\
\<v_1+v_2+v_3,\, v_1+v_2+v_3+v_4\>_M &=& [ \, \|v_1+v_2+v_3\|_M^2+\<v_1+v_2+v_3,v_4\>_M \,].\\
\<2v_1+v_2+v_3+v_4,\, v_1+v_2+v_3+v_4\>_M &=& [ \, \|v_2+v_3+v_4\|_M^2 + 2\<v_2+v_3+v_4,v_1\>\, ] \\
& & + [ \, 2 \|v_1\|_M^2 + \<v_1, v_2\>_M + \<v_1, v_3\>_M + \<v_1, v_4\>_M \, ].
\end{eqnarray*}

\section*{Conclusion}

We have continued in this paper the analysis of a variant of the fast marching algorithm, introduced in \cite{M12}, which is particularly efficient in the context of large anisotropies. 
The computational complexity of this method was known to be largely insensitive to anisotropy.  
This paper establishes, in the special case of a constant metric on the domain $\R^d\sm \{0\}$, $1 \leq d \leq 4$, and in an average sense over all anisotropy orientations, that the numerical accuracy not affected either by anisotropy, however pronounced. A two-dimensional numerical experiment confirms this excellent accuracy, and an extension of the algorithm to four dimensional domains is presented.
Future work will be devoted to the error analysis in the case of a general continuous metric, and to the application of this algorithm to medical image and data analysis.

\paragraph{Acknowledgement : } The author thanks Gabriel Peyr\'e for a constructive discussion, at MIA 2012 conference, which was the starting point of this work.


\begin{thebibliography}{99}

\bibitem{BC10} F. Benmansour, L. D. Cohen, {\it Tubular Structure Segmentation Based on Minimal Path Method and Anisotropic Enhancement}, International Journal of Computer Vision, 92(2), 192-210, 2010.



\bibitem{BR06} F. Bornemann, C. Rasch, {\it Finite-element Discretization of Static Hamilton-Jacobi Equations based on a Local Variational Principle}, Computing and Visualization in Science, 9(2), 57-69, 2006.

\bibitem{JBTDPIB08}  S. Jbabdi, P. Bellec, R. Toro, J. Daunizeau, M. P{\'e}l{\'e}grini-Issac, H. Benali, {\it Accurate Anisotropic Fast Marching for Diffusion-Based Geodesic Tractography}, International Journal of Biomedical Imaging, 2008.

\bibitem{L82} P.L. Lions, {\it Generalized solutions of Hamilton-Jacobi equations}, Pitman, Boston, 1982.

\bibitem{M1896}
H. Minkowski, {\it Geometrie der Zahlen}, Teubner, Leipzig-Berlin, 1896, Reprinted: Johnson, New York, 1968.


\bibitem{M12} J.-M. Mirebeau, {\it Anisotropic Fast Marching on Cartesian Grids, using Lattice Basis Reduction}, preprint, 2012.

\bibitem{NS09}  P. Q. Nguyen, and D. Stehl{\'e}, {\it Low-dimensional lattice basis reduction revisited}, ACM Transactions on Algorithms, Article 46, 2009.

\bibitem{SV03} J. A. Sethian, A. Vladimirsky, {\it Ordered Upwind Methods for Static Hamilton-Jacobi Equations : Theory and Algorithms}, SIAM Journal of Numerical Analysis, 41(1), 325-363, 2003

\bibitem{S99} J. A. Sethian, {\it Level Set Methods and Fast Marching Methods},
J.A. Sethian, Cambridge University Press, 1999.


\bibitem{T95} J. N. Tsitsiklis, {\it Efficient algorithms for globally optimal trajectories}, IEEE Transactions on Automatic Control, 40(9), 1528-1538, 1995.


\end{thebibliography}
\end{document}